\documentclass[11pt]{amsart}
\usepackage{geometry}                
\usepackage{graphicx}
\usepackage{amssymb,amsmath}
\usepackage{epstopdf}
\usepackage{xcolor}
\usepackage{hyperref}

\newcommand{\fex}{f_{ex}}

\newcommand{\gex}{g_{ex}}
\newcommand{\X}{\chi_{\mu}}
\newcommand{\B}{B_\mu}

\newcommand{\bb}{\beta_{\mu}}
\newcommand{\lan}{\langle}
\newcommand{\ran}{\rangle}
\newcommand{\F}{\mathcal{F}}
\newcommand{\G}{\mathcal{G}}

\newtheorem{lemma}{Lemma}
\newtheorem{proposition}{Proposition}
\newtheorem{corollary}{Corollary}

\numberwithin{equation}{section}

\begin{document}

\title[Stability estimates for the truncated Hilbert transform]{Stability estimates for the regularized inversion of the truncated Hilbert transform}
\author[R Alaifari, M Defrise, A Katsevich]{Rima Alaifari$^1$ \and Michel Defrise $^2$ \and Alexander Katsevich$^3$}
\thanks{$^1$ Seminar for Applied Mathematics, ETH Z\"{u}rich, 8092 Z\"{u}rich, Switzerland}
\thanks{$^2$ Department of Nuclear Medicine, Vrije Universiteit Brussel, Brussels B-1050, Belgium}
\thanks{$^3$ Department of Mathematics, University of Central Florida, FL 32816, USA}

\begin{abstract}
In limited data computerized tomography, the 2D or 3D problem can be reduced to a family of 1D problems using the differentiated backprojection (DBP) method. Each 1D problem consists of recovering a compactly supported function $f \in L^2(\mathcal F)$, where $\F$ is a finite interval, from its partial Hilbert transform data. When the Hilbert transform is measured on a finite interval $\mathcal G$ that only overlaps but does not cover $\mathcal F$ this inversion problem is known to be severely ill-posed \cite{adk}.

In this paper, we study the reconstruction of $f$ restricted to the overlap region $\mathcal F \cap \mathcal G$. We show that with this restriction and by assuming prior knowledge on the $L^2$ norm or on the variation of $f$, better stability with H\"older continuity (typical for mildly ill-posed problems) can be obtained.
\end{abstract}
\maketitle

\section{Introduction}

This paper presents new results on the stability of the inversion of the Hilbert transform with limited data. The main application 
concerns tomographic reconstruction from truncated projections, where some 2D and 3D problems can be reduced
to the inversion of the Hilbert transform on a family of line segments defined by the geometry of the scanners. The reduction
of the 2D or 3D tomography problem to a 1D Hilbert transform inversion problem is achieved by backprojecting a derivative of the projection data over an angular range
of 180 degrees. This operation based on a result by Gelfand and Graev \cite{gg}  is referred to as the {\it differentiated backprojection} (DBP). In the 2000's, the DBP triggered a 
remarkable evolution in tomography by allowing accurate reconstruction of 
regions of interest (ROI) from truncated projection data sets \cite{noo,pnc,zp}, which were previously believed to exclude any accurate reconstruction.  
A second method for accurate ROI reconstruction from limited data, the {\it virtual fan-beam algorithm} (VFB), was introduced in 2002; it is based on a different 
mathematical property of the Radon transform and will not be studied in this paper. The reader is referred to \cite{spm} for a review 
on  ROI reconstruction in tomography and the link between that problem and the Hilbert transform.

We work with a line integral model of the tomographic data, assuming continuous (infinitely fine) sampling of the 2D or 3D x-ray
transform of the unknown function $f(x)$ within some subset of the  space of lines in $\mathbb{R}^2$ or $\mathbb{R}^3$. In this context "accurate" reconstruction
means that the data uniquely and stably determine $f(x)$ within a ROI, under appropriate assumptions. These assumptions specify the functional spaces to which
the object and the data belong, the nature and magnitude of the measurement noise, and typically also involve prior constraints on the object. Loosely speaking we call a reconstruction "stable" with \textit{H\"older continuity} if the norm of the error on the estimated $f(x)$ (possibly within some
limited ROI) can be bounded by some positive power of the norm of the measurement error \cite{berterodemolviano}. In contrast one speaks of \textit{logarithmic continuity} when the error bound
decreases only with the logarithm of the
noise. The noise propagation properties of the DBP are well understood, since they are similar to those of the usual 
filtered-backprojection (FBP) algorithm. This can be seen by observing that the two procedures involve the same backprojection, and 
that the ramp filter $|\nu|$ (for FBP) and the
derivative filter $2 \pi i \nu$ (for DBP) similarly amplify high frequencies.
Thus, the major question left is to analyze the stability of the
inversion of the Hilbert transform with limited data. The rest of this paper focuses on that one-dimensional inverse problem. 

To date, much more is known on conditions guaranteeing uniqueness of the solution than on the stability in the presence of noise. 
Stability is easy to analyze when a closed form analytic inversion formula is known, such as the
FBP algorithm for the Radon transform with complete data. With the DBP however, a closed form inversion exists only when 
the limited data allow calculating the Hilbert transform of $f(x)$ on a line segment $\mathcal G$ that contains the support $\mathcal F$ of $f(x)$ 
along that line. In this case, inversion is based on the finite inverse Hilbert transform and has very good stability (all singular values
but one are equal to $1$ in a weighted $L^2$ space \cite{tricomi}).
Unfortunately no closed form inversion is known for the two other configurations: the {\it interior problem} ($\mathcal G \subset
\mathcal F)$ and the {\it truncated problem} where $\mathcal G$ partially overlaps $\mathcal F$. 

It has long been known \cite{nat} that unique inversion is not possible for the interior problem, even within the ROI $\mathcal G$ 
where the data are known. However, uniqueness holds when 
additional information is available. One such instance is the case where $f(x)$ is known in some segment $\mathcal K \subset \mathcal G$ \cite{courdurier,kudo,yeyuwang}. Other results show that uniqueness is restored when $f(x)$ is assumed to belong to some function space (e.g. $f(x)$ is piecewise constant, a polynomial function or a generalized spline) \cite{wardtv,wangtv}. Numerical tests with discretized models of these problems suggest good stability but these results may depend on the way the problem is discretized. To our knowledge explicit stability
estimates have only been obtained when $f(x)$ is a polynomial function in the interior region \cite{k-k-w}.

For the {\it truncated problem} uniqueness of the solution follows from simple analyticity arguments \cite{defrise}.
However, the stability bound obtained in \cite{defrise} is difficult to interpret since it relates the reconstruction error to an upper
bound on the error on an intermediate function and not to the error on the Hilbert transform. Also, the results in \cite{defrise}  do not pertain
to a specific regularization method. Recently, a more explicit stability estimate has been obtained in \cite{aps}. 
It guarantees logarithmic continuity for the reconstruction on the entire support $\mathcal F$ of $f(x)$ if prior knowledge on the total variation of $f$ is assumed.

This paper exploits recent results on the asymptotics of the singular value decomposition (SVD) of the truncated Hilbert transform with
overlap \cite{adk}. With these, explicit $L^2$ stability estimates are obtained. They allow to guarantee stable inversion of the corresponding inverse problem when an a priori bound on the $L^2$ norm of the solution is assumed. We give specific examples for
the truncated SVD and for Tikhonov's method. By seeking to reconstruct only within a \textit{region of interest}, 
H\"older continuity is obtained. A similar result is shown for prior knowledge on the total variation of $f(x)$.

\textbf{Remark.} Stability estimates allow assessing the ill-posedness of an inverse 
problem independently
of the algorithm and independently of the discrete sampling of the data 
and solution. The inversion of the
truncated Hilbert transform is a very specific problem, which is 
severely or mildly ill-posed depending on which
region of interest is reconstructed. The stability estimates in this 
paper are derived in a deterministic framework,
they give worst case error bounds, which are pessimistic and for that 
reason do not provide reliable absolute
values for the regularization parameter. However, our results for the 
Hilbert transform
determine the H\"older power of the reconstruction accuracy in terms of 
known constants depending on the geometry of the
problem (i.e. on the intervals $\F$ and $\G$). This may provide guidance to 
adapt the regularization parameter to varying
geometry or data error, once it has been empirically estimated for one 
case.
Future work will investigate application to image reconstruction from 
incomplete CT data.

\section{The truncated Hilbert transform}

Consider the truncated Hilbert transform $H_T: L^2(\mathcal F) \rightarrow L^2(\mathcal G)$, with
$\mathcal F = (a_2,a_4)$, $\mathcal G = (a_1, a_3)$, and $a_1 < a_2 < a_3 < a_4$:
\begin{equation}
(H_T f)(x) = \frac{1}{\pi} p.v. \int_{a_2}^{a_4} \frac{f(y) dy}{y-x}, \quad x \in (a_1,a_3) 
\label{HT}
\end{equation}
 We define the normalized singular function pair $u_n \in L^2(\F)$, $v_n \in L^2(\G)$ such that $H_T u_n = \sigma_n v_n$ for $n \in \mathbb{Z}$. 
 The singular values $\sigma_n$ are ordered as usual by decreasing values, and the spectrum has two accumulation 
points, $\lim_{n \rightarrow - \infty} \sigma_n = 1$ and $\lim_{n \rightarrow  \infty} \sigma_n = 0$. 
We will make use of the following properties of the singular system of $H_T$:

\begin{itemize}
\item [--] For large $n >0$ the asymptotic behavior of the singular values is given by
\begin{equation}
\sigma_n = 2 e^{-n \pi K_+/K_-} (1+ O(n^{-1/2+\zeta})),\ n\to\infty,
\label{sigmaasympt}
\end{equation}
for some small $\zeta > 0$ and with
\begin{equation}
K_- = \int_{a_1}^{a_2}  \frac{dx}{\sqrt{-P(x)}}, \quad \quad 
K_+ = \int_{a_2}^{a_3}  \frac{dx}{\sqrt{P(x)}}.
\label{KpKm}
\end{equation}
Here, $P(x)=(x-a_1)(x-a_2)(x-a_3)(x-a_4)$.
\item [--] For large negative $n$ the asymptotic behavior of the singular values is
 \begin{equation}
 \sigma_{n} =(1- 2 e^{-2 |n| \pi K_-/K_+}) (1+ O(|n|^{-1/2+\zeta})),\ n\to-\infty.
 \label{slope1}
\end{equation}
\item [--] The singular function $v_n$ is bounded at $a_1$ and $a_3$ and has a logarithmic singularity at $a_2$. For large $n > 0$ it is an oscillatory function in the interval
$(a_1,a_2)$ and a monotonically decreasing function on $(a_2,a_3)$. 
\item [--] The function $u_n$ is bounded at $a_2$ and $a_4$ and has a logarithmic singularity at $a_3$. 
For large $n > 0$ it is a monotonically increasing function in the interval $(a_2,a_3)$ and an oscillatory function on $(a_3,a_4)$. 
\item [--] For sufficiently small $\mu > 0$, the norm of $u_n$ over the segment $(a_2,a_3-\mu)$ has the following asymptotic behavior
\begin{equation}
\left ( \int_{a_2}^{a_3-\mu} dx \, |u_n(x)|^2 \right )^{1/2}= \frac{1}{\sqrt{n \pi}} e^{-\bb n} \, (1+ O(n^{-1/2+ \zeta})),
\label{asymptfn2}
\end{equation}
where
\begin{equation}
\bb = \frac{\pi}{K_-} \int_{a_3-\mu}^{a_3} \frac{dt}{\sqrt{P(t)}} = \frac{ 2 \pi}{K_-} \frac{\sqrt{\mu}}{\sqrt{-P'(a_3)}} (1 + O(\mu)).
\label{betamu}
\end{equation}

\end{itemize}
Except for the last property which is proven in the appendix (see Lemma \ref{lemma1}), the proofs of all the above properties can be found in \cite{adk,a-k}. 
However, the monotonicity of the singular functions $u_n$ is shown more explicitly in the appendix (see Lemma \ref{lemma2}).

\section{Inversion of the truncated Hilbert transform: regularization with an $L^2$ penalty. }

We consider the following inverse problem. Let $\fex \in L^2(\F)$ be an object of interest and let $\gex=H_T \fex \in L^2(\G)$ be the 
corresponding noise-free data, with $H_T$ the truncated Hilbert transform (\ref{HT}). Given a
noisy measurement $g$ such that $\|g-\gex\|_{L^2(\G)} \leq \delta$ for some noise level $\delta > 0$, we seek an estimate of $\fex$ on some interval 
$(a_2,a_3-\mu)$, with some small $\mu > 0$.  

The rationale to restrict the reconstruction to a limited interval 
 is that there is no hope of obtaining a H\"older stability estimate outside the segment $\G=(a_1,a_3)$ where the Hilbert transform is known. 
 Following a similar intuitive analysis by Natterer \cite{nat} for the exterior problem of tomography, this can be seen by considering 
 an object  $\fex$ which is 
 supported in $(a_3,a_4)$ and has a discontinuity somewhere in that interval. For such an object, one sees from (\ref{HT}) that
  $H_T \fex$ is $C^\infty$ and therefore an inverse
 operator should map a $C^\infty$ function onto a discontinuous function, which corresponds to a severely ill-posed 
 problem. This severe ill-posedness is also expected from the exponential decay of the singular values, equation (\ref{sigmaasympt}). Thus, unless
 very restrictive prior knowledge is assumed the quest for a stable 
 reconstruction must be restricted to an interval such as $(a_2,a_3-\mu)$. By varying the  parameter $\mu$ one can study the 
 degradation of the stability as one gets closer to the limit of  the stably recoverable region $(a_2,a_3)$. 
  We denote by $\X$  the characteristic function of that "region-of-interest" interval $(a_2,a_3-\mu)$. 
  
We assume the following properties of the SVD:
\begin{itemize}
\item [--] P1: $0 < \sigma_n < 1$.
\item [--] P2: There are some $N_0 \in \mathbb{N}$ and positive real $A, \alpha$ such that $\sigma_n \geq A e^{- \alpha n}$ for $n > N_0$.
\item [--] P3: There are some $N_\mu \in \mathbb{N}$ and positive real $\B , \bb$ such that $\|\X u_n\|_{L^2(\F)} \leq \B e^{- \bb n}$ for $n > N_\mu$. 
\end{itemize}
The constants $A$ and $\alpha$ can be obtained from (\ref{sigmaasympt}):
\begin{equation}
A < 2, \quad \quad 
\alpha= \pi K_+/K_-, 
\label{Aalpha}
\end{equation}
where $A=2$ corresponds to the leading term of the asymptotic behavior of the singular values. Hence, any value of $A$ 
smaller than $2$ provides an upper bound for sufficiently large $n$. We note that the index $N_\mu$ and the constants $B_\mu, \beta_\mu$ all depend on the choice of $\mu$. In particular, for P3 to hold, $N_\mu = O(\mu^{-\frac{1}{1+2\zeta}})$; so the smaller $\mu$ is chosen to be, the larger $N_\mu$ has to be taken. Throughout this paper, we assume $\mu>0$ to be arbitrarily small but fixed. The results obtained in the following sections are no longer valid as $\mu \to 0$. This is due to the fact that $\|\chi_{\mu=0} u_n\|_{L^2(\F)}^2$ only decays as $O(1/n)$ (see equation (8.6) in \cite{adk}).
In what follows, we select $N_\mu$ such that $N_\mu > N_0$, so that P2 holds for all $n> N_\mu$. Property P3 follows from equation (\ref{asymptfn2}), 
this asymptotic relation can be changed into an upper bound by taking for instance
\begin{equation*}
\B=\frac{1}{\sqrt{N_\mu \pi}}. 
\end{equation*}

\subsection{An algorithm independent stability bound.}\label{ind-stab}
We assume as regularizing prior knowledge on the solution an upper bound $E > 0$ on the norm of the object, $\|\fex \|_{L^2(\F)} \leq E$. 
The stability bound that we present follows the approach of Miller \cite{miller}. This formulation is independent of a specific algorithm, but instead yields an 
upper bound on the $L^2$ distance (within the interval $(a_2,a_3-\mu)$) between any pair of solutions that are compatible both with the data and with the prior knowledge.

We define the set of {\it admissible solutions} as
\begin{equation*}
{\mathcal S} = \left \{ f \in L^2(\F) \ {\Big | } \ \|H_Tf-g\|_{L^2(\G)} \leq \delta \mbox{   and   } \|f\|_{L^2(\F)} \leq E \right \}.
\end{equation*}
\begin{proposition}\label{prop1}
Let $H_T \fex = \gex$ and assume $\| \fex \|_{L^2{(\F)}} \leq E$. We consider the reconstruction of $\fex$ from a noisy measurement $g$ for which $\|g-\gex\|_{L^2(\G)} \leq \delta$ is given. 
Then, any algorithm that for given $\delta$, $E$ and $g$ computes a solution in $\mathcal S$, is a regularization method for the reconstruction on $(a_2,a_3-\mu)$. 
More precisely, given any two admissible solutions $f_a, f_b \in \mathcal S$, one has for sufficiently small $\delta$ the bound
 \begin{equation}
 \| \X (f_a - f_b) \|_{L^2(\F)} \leq \frac{ 2 \delta e^{\alpha N_\mu}}{A} + 2 E \B \left \{ \frac{\delta}{A V_\mu E} \right \}^{\bb / \alpha} \, \left (\frac{\alpha}{(\alpha - \bb) \, \sqrt{e^{2 \bb}-1}} \right ),
\label{boundf1f2}
\end{equation}
where $V_\mu$ is a constant, which only depends on $\alpha$ and $\bb$.
\end{proposition}
\begin{proof} Consider two admissible solutions $f_a, f_b \in {\mathcal S}$. 
The corresponding Hilbert transforms $g_a=H_T f_a$ and $g_b = H_T f_b$ satisfy $\|g_a - g_b\|_{L^2(\G)} \leq \|g_a-g\|_{L^2(\G)}+\|g_b-g\|_{L^2(\G)} \leq 2 \delta$. 
Similarly $\|f_a-f_b\|_{L^2(\F)} \leq 2E$. We can rewrite $f_a-f_b$ using the SVD of $H_T$ and splitting the obtained expansion into three terms as follows:
\begin{equation}
f_a - f_b = \sum_{n=-\infty}^{N_\mu} \lan g_a-g_b, v_n\ran  \, \frac{1}{\sigma_n} \, u_n + \sum_{n=N_\mu+1}^{N} \lan g_a-g_b, v_n\ran  \, \frac{1}{\sigma_n} \, u_n + \sum_{n=N+1}^\infty \lan f_a-f_b, u_n\ran  \,  u_n.
\label{f1mf2}
\end{equation}
Here, the SVD series is split at $N_\mu$ and at some -- so far arbitrary -- index $N$ and we used $\lan  g_a, v_n \ran= \lan  H_Tf_a,v_n \ran= \sigma_n \lan f_a,u_n \ran$ 
and idem for $f_b$. Equation (\ref{f1mf2}) holds for any $N > N_\mu$.

We seek to derive an upper bound on the error in the interval of interest in the form of
\begin{equation}
\| \X (f_a - f_b)\|_{L^2(\F)} \leq I_1 + I_2 + I_3
\label{bound0}
\end{equation}
where $I_1,I_2,I_3$ correspond to the norms of the three terms in (\ref{f1mf2}) restricted by $\X$.

The error due to the large singular values is bounded as
\begin{align}
I_1^2 &= \| \X \sum_{n=-\infty}^{N_\mu} \lan g_a - g_b, v_n\ran  \, \frac{1}{\sigma_n} \, u_n \|_{L^2(\F)}^2 \leq \| \sum_{n=-\infty}^{N_\mu}  \lan g_a - g_b, v_n\ran  \, \frac{1}{\sigma_n} \, u_n \|_{L^2(\F)}^2 
\nonumber \\
&\leq   \frac{1}{\sigma^2_{N_\mu}} \sum_{n=-\infty}^{N_\mu}  |  \lan g_a - g_b, v_n\ran  |^2  \leq \frac{(2\delta)^2}{\sigma^2_{N_\mu}}. 
\label{I1}
\end{align}

The contribution of the intermediate terms corresponding to $N_\mu < n \leq N$ to the error is bounded as
\begin{align}
I^2_2 &= \| \sum_{n=N_\mu+1}^N \lan g_a-g_b, v_n\ran  \, \frac{1}{\sigma_n} \,  \X u_n \|_{L^2(\F)}^2 \leq 
\left \{ \sum_{n=N_\mu+1}^N \, |  \lan g_a-g_b, v_n\ran | \, \frac{1}{\sigma_n} \, \|  \X u_n \|_{L^2(\F)} \right \}^2 \nonumber \\
&\leq   \sum_{n'=N_\mu+1}^N \, |  \lan g_a-g_b, v_{n'}\ran |^2   \sum_{n=N_\mu+1}^N \frac{1}{\sigma^2_n} \, \|  \X u_n \|_{L^2(\F)}^2    
\leq (2\delta)^2 \sum_{n=N_\mu+1}^N \,\frac{1}{\sigma^2_n} \,   \|  \X u_n \|_{L^2(\F)}^2 \nonumber \\
&\leq \frac{ (2\delta)^2 \, \B^2}{A^2} \, \sum_{n=N_\mu+1}^N \, e^{2 (\alpha-\bb) n} = \frac{ (2\delta)^2 \, \B^2}{A^2} \,  
 \frac{ e^{2 (\alpha-\bb) N} -  e^{2 (\alpha-\bb) N_\mu} }{ 1 -  e^{-2 (\alpha-\bb)} }   \nonumber \\
&\leq \frac{ (2\delta)^2 \, \B^2}{A^2} \, \frac{ e^{2 (\alpha-\bb) N}}{ 1 -  e^{-2 (\alpha-\bb)} }, \label{I2}
\end{align}
where we successively used the triangle inequality, Schwarz inequality, and majorized by extending the sum over $n'$ to all $\mathbb{Z}$.
Recall from the definitions of these constants that $\alpha > \bb > 0$.
The last inequality is not strictly needed but simplifies subsequent derivations, and it is expected to have a limited impact since for small 
$\delta$, we will find that the optimal cut-off satisfies $N \gg N_\mu$. 

Finally,
\begin{align}
I^2_3 &= \| \sum_{n=N+1}^\infty \lan f_a-f_b, u_n\ran  \,  \X u_n \|_{L^2(\F)}^2 \leq \left \{ \sum_{n=N+1}^\infty \, | \lan f_a-f_b, u_n\ran | \, \|  \X u_n \|_{L^2(\F)} \right \}^2 \nonumber \\
&\leq   \sum_{n'=N+1}^\infty \, |  \lan f_a-f_b, u_{n'}\ran    |^2   \sum_{n=N+1}^\infty  \|  \X u_n \|_{L^2(\F)}^2    
\leq \|f_a-f_b\|_{L^2(\F)}^2 \sum_{n=N+1}^\infty \,  \|  \X u_n \|_{L^2(\F)}^2 \nonumber \\
&\leq (2 E)^2 \, \B^2 \sum_{n=N+1}^\infty \, e^{-  2 \bb n} = \left \{ 2 E \, \B e^{- \bb N} /\sqrt{e^{2 \bb}-1} \right \}^2. \label{I3}
\end{align}

Inserting equations (\ref{I1}), (\ref{I2}), (\ref{I3}) into (\ref{f1mf2}) leads to the following upper bound on the $L^2$ distance between $f_a$ and $f_b$ on the region-of-interest interval $(a_2, a_3-\mu)$:
\begin{equation}
 \| \X (f_a - f_b) \|_{L^2(\F)} \leq  \frac{2 \delta}{\sigma_{N_\mu}}   +  \frac{2 \delta \B}{A} \, 
\frac{ e^{ (\alpha-\bb) N} }{ \sqrt{1 -  e^{-2 (\alpha-\bb)} } }
+ \frac{2 E  \B e^{- \bb N}}{\sqrt{e^{2 \bb}-1}}.
\label{boundmse}
\end{equation}
A quasi-optimal splitting index $N(\delta)$ is obtained by treating $N$ as a real variable\footnote{We hereby neglect the fact that $N$ only takes integer values. This hardly changes the 
error estimate when $\delta$ is small.} and by minimizing the convex function in the RHS of (\ref{boundmse}) w.r.t. $N$:
\begin{equation}
N(\delta) = \frac{1}{\alpha} \, \log \left (  \frac{E A V_\mu}{\delta} \right )
\label{optmseN}
\end{equation}
with the constant
\begin{equation*}
V_\mu = \frac{ \bb}{(\alpha - \bb)}  \,  \left \{ \frac{ \left ( 1- e^{-2 (\alpha-\bb)} \right )}{ \left (e^{2 \bb} -1 \right ) }  \right \}^{1/2}.
\end{equation*}

Inserting $N(\delta)$ into (\ref{boundmse}) finally yields (\ref{boundf1f2}). This error bound is valid when $N(\delta)  > N_\mu$, a condition which is always satisfied for sufficiently small noise level $\delta$. 

To conclude the proof, consider any algorithm that produces for each $\delta, E$ and $g$ a solution $f_a  \in {\mathcal S}$. The bound (\ref{boundf1f2}) can be
applied with $f_b = f_{ex}$ because $f_{ex} \in {\mathcal S}$, and since the right-hand side in (\ref{boundf1f2}) tends to zero as $\delta \to 0$, this
algorithm is a regularizing method for the reconstruction on $(a_2,a_3-\mu)$.
\end{proof}
\textbf{Remarks.}
\begin{itemize}
\item[--] For small $\delta$ the second term in (\ref{boundf1f2}) eventually dominates. 
An intuitive
interpretation of the stability estimate (\ref{boundf1f2}) is that the number of significant digits that can be reliably recovered when
estimating the solution is roughly equal to a fraction $\bb/\alpha$ of the number of significant digits in the
measured data.
\item[--] The first sum in (\ref{f1mf2}) starts at $-\infty$ because the spectrum of $H_T$ has two accumulation points. 
We remark that the singular components up to $n \rightarrow -\infty$ 
can be stably recovered from the data because $\sigma_n \rightarrow 1$ as $n \rightarrow -\infty$. 
This holds only for the continuous-continuous model of the inverse problem, which is analyzed in this paper. 
In practice, only sampled data are available. Since the functions $u_n$ are increasingly oscillating functions on the interval $(a_2, a_3)$ when $n \rightarrow -\infty$ \cite{adk,a-k}, one would in
practice start the SVD expansion at a large negative $N_{min} < 0$. This lower SVD cut-off would have an effect similar to a cut-off on the high spatial frequencies.
\item[--] Setting $g=f_b=0$ leads to another equivalent formulation of stability:
if $f \in L^2(\F)$ is such that $\|H_Tf\|_{L^2(\G)} \leq \delta$ and $\|f\|_{L^2(\F)} \leq E$, then
 \begin{equation}
 \| \X f \|_{L^2(\F)}  \leq \frac{\delta e^{\alpha N_\mu}}{A}+ E \B \left \{ \frac{\delta}{A V_\mu E} \right \}^{\bb / \alpha} \, \left (\frac{\alpha}{(\alpha - \bb) \, \sqrt{e^{2 \bb}-1}} \right ).
\label{boundf1}
\end{equation}
\end{itemize}

The following corollaries of Proposition \ref{prop1} give explicit error bounds for two specific regularization algorithms. The proofs follow the same line as the proof of Proposition \ref{prop1} and are therefore omitted.

 \begin{corollary}[Truncated SVD]
The truncated SVD estimate
\begin{equation}
f_{N(\delta)} = \sum_{n=-\infty}^{N(\delta)} \lan g, v_n\ran  \, \frac{1}{\sigma_n} \, u_n
\label{TSVD}
\end{equation}
with the cut-off index $N(\delta)$ defined by (\ref{optmseN}) guarantees regularization in the sense that
\begin{equation}
\lim_{\delta \rightarrow 0} \| \X (f_{N(\delta)} - \fex)\|_{L^2(\F)}=0.
\label{regultsvd}
\end{equation}
Specifically, if $N(\delta) > N_\mu$,
\begin{equation}
 \| \X (f_N - \fex) \|_{L^2(\F)} \leq \frac{\delta e^{\alpha N_\mu}}{A} 
+ E \B \left \{ \frac{\delta}{A V_\mu E} \right \}^{\bb / \alpha} \, \left (\frac{\alpha}{(\alpha - \bb) \, \sqrt{e^{2 \bb}-1}} \right ).
\label{boundmse1}
\end{equation}
\end{corollary}

Note that the quasi-optimal number (\ref{optmseN}) of terms in the truncated SVD (\ref{TSVD}) increases only very slowly with decreasing noise level $\delta$,
which is a typical behavior of severely ill-posed inverse problems. 
The corresponding lower bound on the singular values that are included in the truncated SVD solution estimate,
\begin{equation}
\sigma_{N(\delta)} \geq  \frac{\delta}{E V_\mu}
\label{sigmaoptmseN}
\end{equation}
is proportional to the noise to signal ratio $\delta/E$. This behavior of the cut-off singular value is 
also obtained for inverse problems with compact operators using Miller's method \cite{berterodemolviano,miller}, \cite[p. 258]{bertero}. 

\begin{corollary}[Tikhonov regularization]
The Tikhonov estimate 
\begin{equation}
f_\eta = \arg \min_{ f \in L^2(\F)} \Phi_\eta(f)   \mbox{  with  } \Phi_\eta(f)= \|H_Tf-g\|_{L^2(\G)}^2 + \eta \|f\|_{L^2(\F)}^2 
\label{Tikhonov}
\end{equation}
with $\eta=\eta(\delta)>0$ such that $\eta(\delta) \rightarrow 0$ and 
that $\delta^2/\eta(\delta)$ remains bounded as $\delta \rightarrow 0$, guarantees regularization in the sense that
\begin{equation*}
\lim_{\delta \rightarrow 0} \| \X (f_{\eta(\delta)} - \fex)\|_{L^2(\F)}=0.
\end{equation*}
In particular the choice $\eta = \delta^2/E^2$ leads to
\begin{align}
 \| \X (f_\eta - \fex) \|_{L^2(\F)}  &\leq (1+\sqrt{2}) \frac{\delta e^{\alpha N_\mu}}{A} \nonumber \\
&+ (1+\sqrt{2}) E \B \left \{ \frac{\delta}{A V_\mu E} \right \}^{\bb / \alpha} \, \left (\frac{\alpha}{(\alpha - \bb) \, \sqrt{e^{2 \bb}-1}} \right ).
\label{boundTikho}
\end{align}
provided $\delta$ is small enough such that $N(\delta) > N_\mu$ in (\ref{optmseN}).
\end{corollary}

The exponential decay of the singular values of the truncated Hilbert transform (see equation (\ref{sigmaasympt})) is typical of severely ill-posed
inverse problems. Equations  (\ref{boundf1f2}, \ref{boundmse1}, \ref{boundTikho}), however, give a H\"older continuity expected for mildly ill-posed problems. This mild ill-posedness
is due to the fact that the error is only calculated over a restricted interval $(a_2,a_3-\mu)$, which excludes a neighborhood of the
limit $a_3$ of the segment where the Hilbert transform is known. The H\"older  error bound does not hold in the limit $\mu \rightarrow 0$; the squared norm $\|  \chi_{\mu=0} u_n \|_{L^2(\F)}^2$ decays as $O(1/n)$ (equation (8.6) in \cite{adk}), 
too slowly to balance the exponential decay of $\sigma_n$.

The error bounds obtained have the same dependence on $\delta^{\bb/ \alpha}$. This behavior is determined essentially by the exponential decay rates $\alpha$ and $\bb$, but also by the assumed regularizing prior knowledge $\| f \|_{L^2(\F)} \leq E$. Figure \ref{fig2} illustrates the change in the rate $\delta^{\bb/ \alpha}$ as the amount of overlap between the two intervals $\mathcal F$ and $\mathcal G$ is varied. In the following section, we seek to find the H\"older exponent using a non-quadratic regularizing penalty, or more precisely, a total variation (TV) penalty. 
It turns out that this yields the same dependence on $\delta^{\bb/ \alpha}$.
\begin{figure}[h!]
\begin{center}
\includegraphics[angle=0,width=0.45\linewidth]{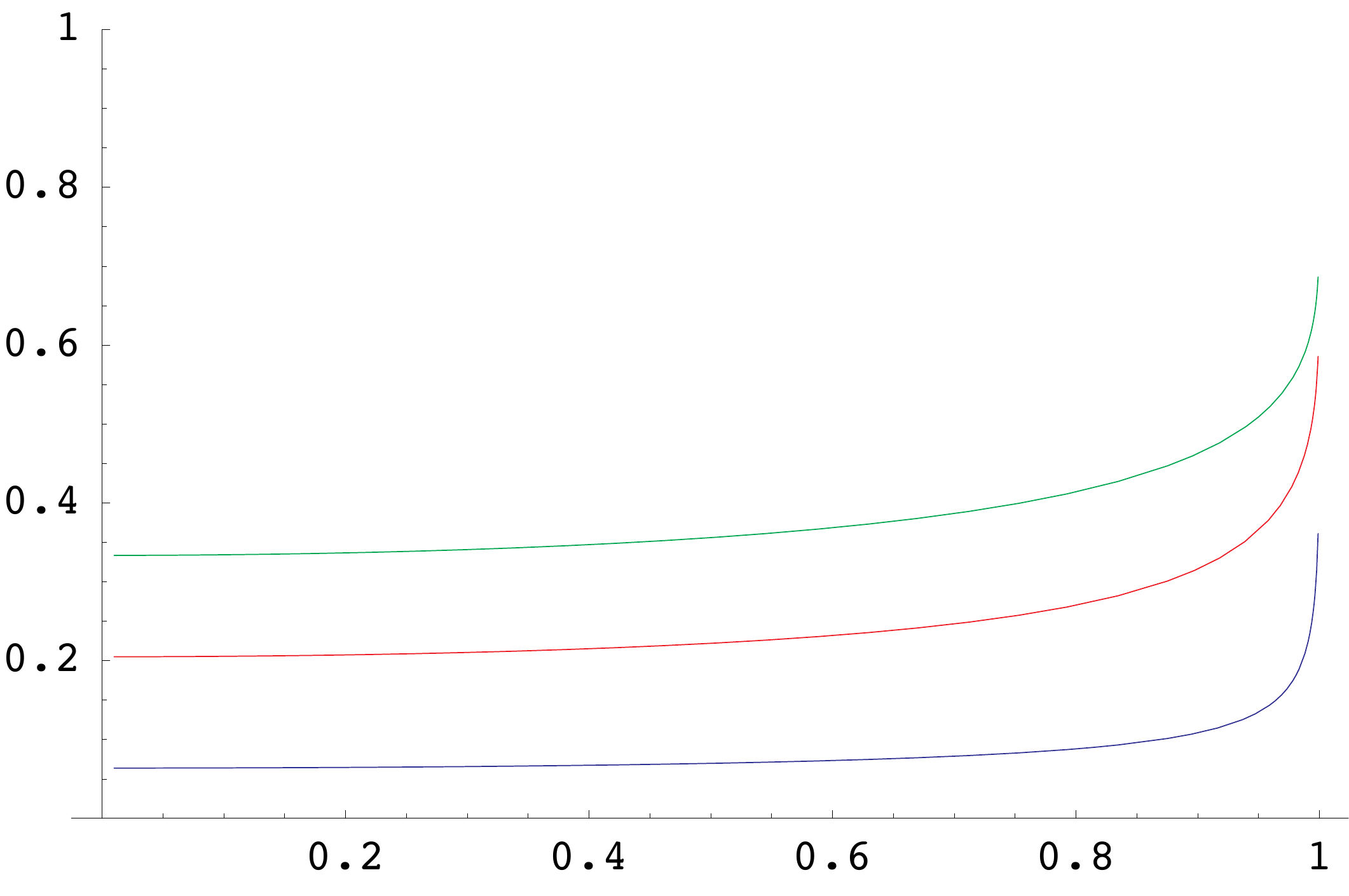}
\caption{\footnotesize The H\"older power $\beta_\mu/\alpha$ is plotted for $0 < a_3 < 1$, and for fixed
$a_1=-1,a_2=0,a_4=1$ to illustrate the dependence of the convergence rate of the
regularized solution on the size of the overlap region $(a_2,a_3)$ relative to the function support $(a_2,a_4)=(0,1)$. The green, red and blue curves correspond to $\mu= 0.25 \cdot a_3, 0.1 \cdot a_3$ and $0.01 \cdot a_3$.}
\label{fig2}
\end{center}
\end{figure}
\section{Inversion of the truncated Hilbert transform: Stability through a TV penalty}

In this section, we drop the condition $\|f_{ex}\|_{L^2(\F)}\leq E$ and instead consider regularization of the inversion problem by assuming prior knowledge on the variation of $f_{ex}$ in the form of the upper bound $|\fex|_{TV} \leq \kappa$ for some $\kappa > 0$. Additionally, we assume that $f_{ex}$ vanishes at $a_2$ and $a_4$. 
This latter assumption is not too restrictive; for if $\fex$ does not vanish at the boundaries, we can always artificially enlarge the interval $\F = (a_2,a_4)$ such that the support of $\fex$ is a strict subset of it. 

We define the set of {\it admissible solutions} as
\begin{equation}
\mathcal{S} = \{ f \in BV(\F): \| H_T f - g \|_{L^2(\G)} \leq \delta, \ |f|_{TV} \leq \kappa \text{ and } f \text{ vanishes at } a_2, a_4\}.
\label{setadmissibletv}
\end{equation}
\begin{proposition}\label{prop2}
Let $H_T \fex = \gex$ and assume $|\fex|_{TV} \leq \kappa$ and $\fex(a_2)=\fex(a_4)=0$. 
We consider the reconstruction of $\fex$ from a noisy measurement $g$ for which $\|g-\gex\|_{L^2(\G)} \leq \delta$ is given. 
Then, any algorithm that for given $\delta$, $\kappa $ and $g$ computes a solution in $\mathcal S$, is a regularization method for the reconstruction on $(a_2,a_3-\mu)$. 
More precisely, given any two admissible solutions $f_a, f_b \in \mathcal S$ and for sufficiently small $\delta$, one has the following bound:
\begin{equation}
\| \chi_\mu (f_a-f_b)\|_{L^2(\F)} \leq 2 \delta A^{-1} e^{\alpha N_\mu} + \frac{2 c}{N_\mu} B_\mu \kappa^{\frac{\alpha-\beta_\mu}{\alpha}} \Big( \frac{\delta}{A W_\mu}\Big)^{\frac{\beta_\mu}{\alpha}} \frac{\alpha}{(\alpha-\beta_\mu)(e^{\beta_\mu}-1)},\label{roi-stab-tv}
\end{equation}
Here, $c$ is a constant only dependent on $a_1, \dots, a_4$ and $W_\mu$ is a constant only dependent on $\alpha, \bb, c$ and $N_\mu$.
\end{proposition}

\begin{proof}
Let $g_a = H_T f_a$ and $g_b = H_T f_b$. Then, we can write the same identity as in (\ref{f1mf2}) which is exact for any choice of $N > N_\mu$. 
Lemma 6.1 in \cite{aps} implies the following decay rate for sufficiently large index $n$: there exists a constant $c$ only dependent on $a_1, \dots, a_4$ such that for $n > N_0$,
$$| \langle f_a-f_b, u_n \rangle| \leq c \frac{|f_a-f_b|_{TV}}{n}.$$
Similarly to the derivations in the previous section, we can obtain an upper bound on $\| \chi_\mu (f_a-f_b) \|_{L^2(\F)} \leq I_1+I_2+I_3$, where $I_1$ and $I_2$ are bounded as in (\ref{I1}) and (\ref{I2}), respectively. An upper bound on the last term is obtained as follows:
\begin{align*}
I_3 &= \| \sum_{n=N+1}^{\infty} \langle f_a - f_b, u_n \rangle \chi_\mu u_n \|_{L^2(\F)}  \leq  \sum_{n=N+1}^{\infty} | \langle f_a - f_b, u_n \rangle | \| \chi_\mu u_n \|_{L^2(\F)} \\
& \leq  c |f_a-f_b|_{TV} \sum_{n=N+1}^{\infty} \frac{1}{n} \| \chi_\mu u_n\|_{L^2(\F)}  \leq c |f_a-f_b|_{TV} B_\mu  \sum_{n=N+1}^{\infty} \frac{e^{-\beta_\mu n }}{n} \\
& \leq  \frac{2 c \kappa B_\mu e^{- \beta_\mu N}}{N_\mu(e^{\beta_\mu}-1)}. 
\end{align*}

The last inequality in the above is not sharp but simplifies subsequent derivations. With this we obtain
\begin{align*}
\| \chi_\mu(f_a-f_b)\|_{L^2(\F)} & \leq \frac{2 \delta}{\sigma_{N_\mu}} +\frac{2 \delta B_\mu}{A}  \frac{e^{(\alpha-\beta_\mu)N}}{(1-e^{-2(\alpha-\beta_\mu)})^{1/2}} + \frac{2 c \kappa B_\mu e^{- \beta_\mu N}}{N_\mu(e^{\beta_\mu}-1)}. 
\end{align*}
The value $N$ that minimizes this upper bound is found to be
\begin{equation*}
N(\delta) = \frac{1}{\alpha} \log \Big( \frac{\kappa A W_\mu}{\delta}\Big),
\end{equation*}
where
$$W_\mu = \frac{\beta_\mu}{\alpha-\beta_\mu}\frac{(1-e^{-2(\alpha-\beta_\mu)})^{1/2}c }{N_\mu (e^{\beta_\mu}-1)}.$$
This choice yields (\ref{roi-stab-tv}). This bound is valid if $N(\delta) > N_\mu$, which can be reformulated as an upper bound on the ratio $\delta/\kappa$:
\begin{equation}
\frac{\delta}{\kappa} < A W_\mu e^{-\alpha N_\mu}. \label{ntv-ratio}
\end{equation} 
To conclude the proof, consider any algorithm that for each $\delta, \kappa$ and $g$ produces a solution $f_a  \in {\mathcal S}$. The bound (\ref{roi-stab-tv}) can be
applied with $f_b = f_{ex}$ because $f_{ex} \in {\mathcal S}$, and since the right-hand side in (\ref{roi-stab-tv}) tends to zero as $\delta \to 0$, this
algorithm is a regularizing method for the reconstruction on $(a_2,a_3-\mu)$.
\end{proof}
\textbf{Remark.} The bounds in (\ref{roi-stab-tv}) together with (\ref{ntv-ratio}) imply that the error within the ROI tends to zero as $\kappa \to 0$ even for constant noise level $\delta$. This is only due to the prior assumption that $\fex$ vanishes at $a_2$ and $a_4$. Together with $|\fex|_{TV}=0$ this assumption implies that $f \equiv 0$. 
Any stability estimate that only requires an upper bound on $|f_{ex}|_{TV}$ without assuming that $\fex$ vanishes at the boundaries, will not tend to zero as $\kappa \to 0$. 

\subsection{Reconstruction on the entire interval}

So far, we have found stability estimates for the reconstruction on the overlap region (or region of interest) while stable reconstruction outside of the region of interest has not been guaranteed. One might ask whether it is possible to guarantee stable reconstruction on all of $\mathcal F = (a_2,a_4)$. If the assumed prior knowledge is of the form  $\|\fex\|_{L^2(\F)} \leq E$, such an estimate cannot hold in general. This can be seen for example by taking the sequence of singular functions $u_n$ for which $\|u_n\|_{L^2(\F)}=1$ while $\|H_T u_n\|_{L^2(\G)} \to 0$.

In \cite{aps} it has been shown that stability on all of $\F = (a_2,a_4)$ can be guaranteed if the variation of the solution is assumed to be bounded, i.e. $|f|_{TV} \leq \kappa$, for some $\kappa>0$. The estimate obtained is of logarithmic continuity (as opposed to H\"{o}lder continuity). This is typical for severely ill-posed problems.

We seek to derive such an estimate employing the methods applied in the previous sections. More precisely, we find an upper bound on $\|f_a-f_b\|_{L^2(\F)}$ for $f_a, f_b \in \mathcal{S}$ where $\mathcal{S}$ is the set of admissible solutions in (\ref{setadmissibletv}). For this, we write
\begin{align*}
\|f_a-f_b\|_{L^2(\F)} \leq I_1 + I_2,
\end{align*}
where
\begin{align*}
I_1 &= \| \sum_{n=-\infty}^N \langle g_a-g_b, v_n\rangle \frac{1}{\sigma_n} u_n \|_{L^2(\F)}, \qquad I_2 = \| \sum_{n=N+1}^\infty \langle f_a-f_b, u_n \rangle u_n \|_{L^2(\F)},
\end{align*}
and $N > N_0$. As before, $I_1 \leq \frac{2\delta}{\sigma_N}$.
For an upper bound on $I_2$, we again apply Lemma 6.1 in \cite{aps} to obtain
\begin{align*}
I_2^2 = \sum_{n=N+1}^\infty |\langle f_a-f_b, u_n \rangle|^2 \leq c^2 (2 \kappa)^2 \sum_{n=N+1}^\infty \frac{1}{n^2} \leq c^2 (2 \kappa)^2 \frac{1}{N}.
\end{align*}
This yields
\begin{equation*}
\|f_a-f_b\|_{L^2(\F)} \leq 2 \delta A^{-1} e^{\alpha N} + 2 \kappa c \frac{1}{\sqrt{N}}.
\end{equation*}
The value $N=N(\delta)$ that minimizes the right-hand side in the above satisfies
$$2 \delta A^{-1} \alpha e^{\alpha N(\delta)} - \kappa c N(\delta)^{-3/2} = 0$$
and therefore
$$\alpha N(\delta) + \frac{3}{2} \log N(\delta) = \log \Big( \frac{A \kappa c}{2 \delta \alpha}\Big).$$
As a consequence, the optimal $N(\delta)$ satisfies
$$(\alpha+\frac{3}{2}) N(\delta) \geq \log \Big( \frac{A \kappa c}{2 \delta \alpha}\Big) \geq \alpha N(\delta)$$
and hence,
$$N(\delta)^{-\frac{1}{2}} \leq (\alpha+\frac{3}{2})^{\frac{1}{2}}  \Big[ \log \Big( \frac{A \kappa c}{2 \delta \alpha}\Big) \Big]^{-\frac{1}{2}}.$$
Thus,
\begin{equation}
\| f_a-f_b\|_{L^2(\F)} \leq \frac{\kappa c}{\alpha} N(\delta)^{-\frac{3}{2}} + 2 \kappa c N(\delta)^{-\frac{1}{2}} \leq \kappa c (\frac{1}{\alpha}+2) N(\delta)^{-\frac{1}{2}} 
\leq \kappa C \big[ \log(\frac{\kappa}{\delta}) + D \big]^{-\frac{1}{2}} \label{stab-tv}
\end{equation}
for some constants $C>0$ and $D$ that depend on $a_1, a_2, a_3, a_4$. This bound holds for $N(\delta) > N_0$, for which
$$\frac{\delta}{\kappa} < \frac{A c}{2 \alpha} e^{-(\alpha+\frac{3}{2})N_0}$$
is a sufficient condition. Note that by this requirement, $\log(\frac{\kappa}{\delta}) + D > 0$.

\section{Numerical validation of the asymptotic expressions in (\ref{asymptfn2})}

We have calculated the SVD of the truncated Hilbert transform with $a_1=0, a_2=450,a_3=1350, a_4=1725$. The problem has been discretized with a sampling step equal to $1$ and a shift of $1/2$ between the object and data samples. Mathematica (using $20$ significant digits) finds 
$910$ nonzero singular values for the discretized $1351 \times 1276$ Hilbert matrix ${\bf H_T}$. As observed in \cite{adk}, most singular values are very close to $1$, and only 10 singular values are smaller than $0.97$; of those the last nine are smaller than $0.01$. These last singular values are well fitted by the 
asymptotic expression $A e^{- \alpha n}$ (Figure 2a) with the constants from equation (\ref{KpKm}). 

For each singular function $u_n$ corresponding to these last $9$ values, we have computed the norm $\| \X u_n\|_{L^2(\F)}$ in 
the interval $(a_2,a_3- \mu)$, for integer values of $\mu$ between $1$ and $400$. For all $\mu$, we find that $\| \X u_n\|_{L^2(\F)}$ 
is well fitted by the asymptotic expression (\ref{asymptfn2}) as illustrated in Figure 2b. The fits are obtained by associating $n=1$ to the 
singular value $n'=902$ of the discrete SVD.

\begin{figure}[h!]
\begin{center}
\begin{tabular}{c c}
\includegraphics[angle=0,width=0.45\linewidth]{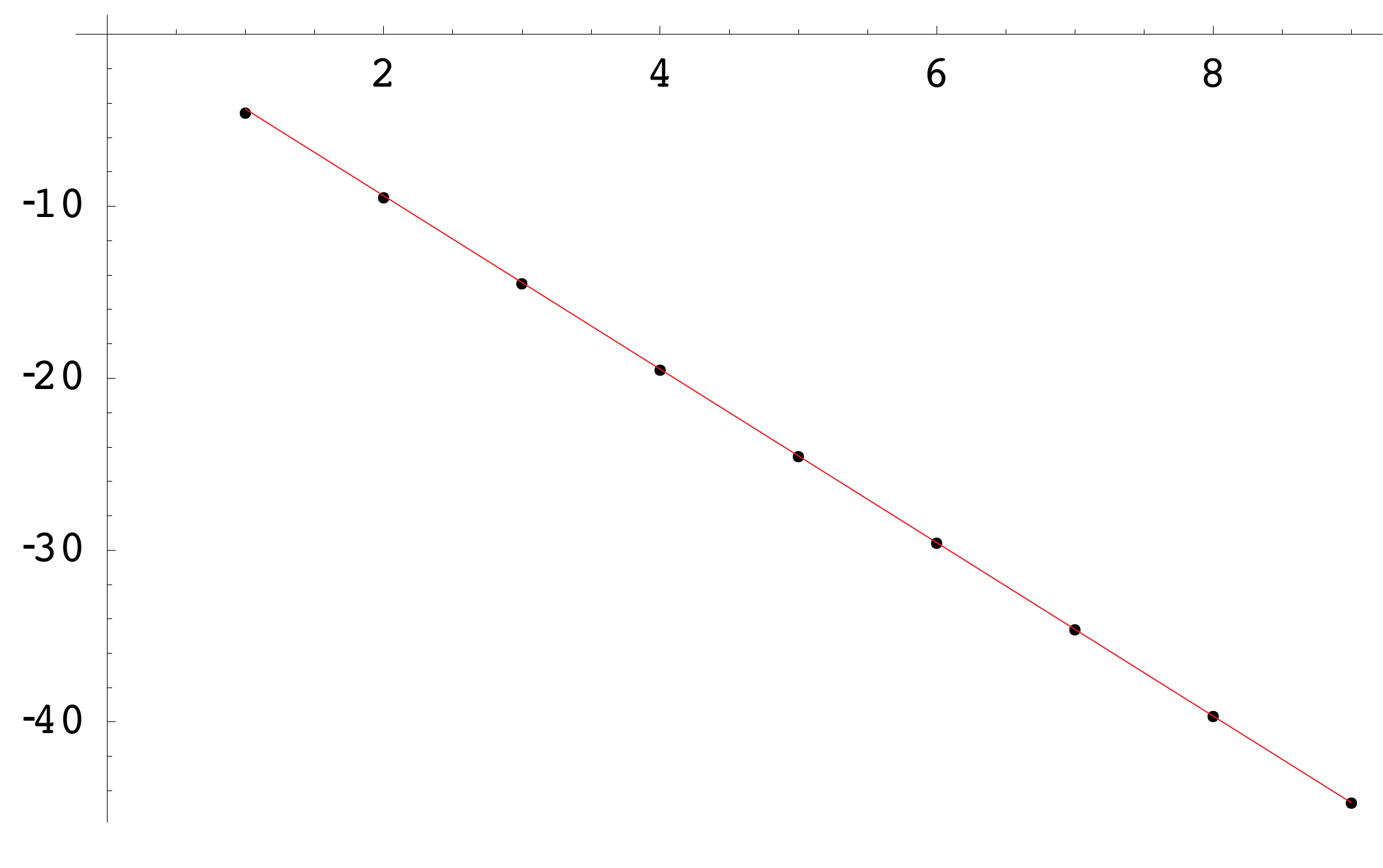} &
\includegraphics[angle=0,width=0.45\linewidth]{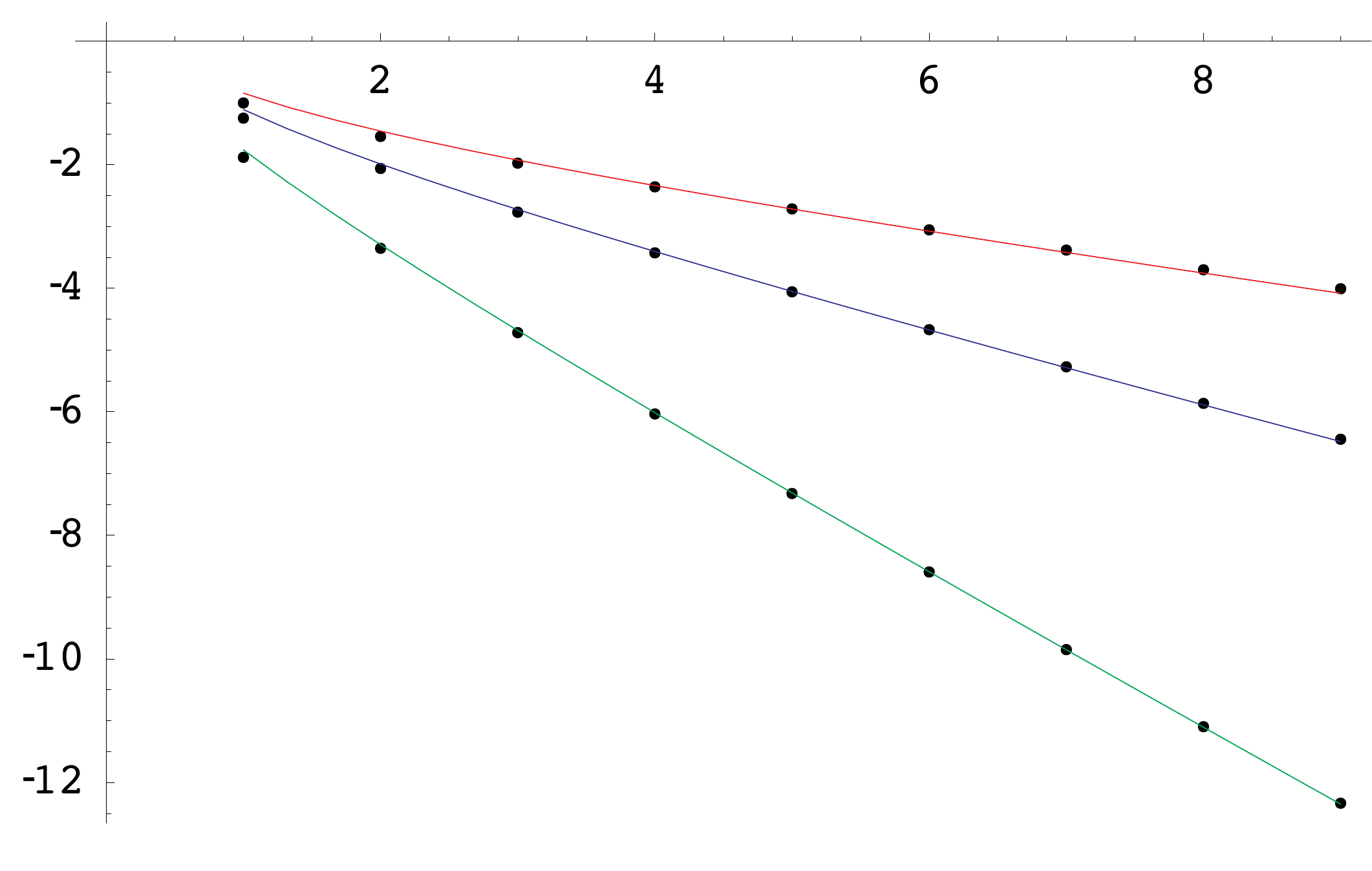}
\end{tabular}
\caption{\footnotesize {\it Left}: $\log \sigma_n$ versus $n$ for the discretized truncated Hilbert 
problem (dots) with $a_1=0, a_2=450,a_3=1350, a_4=1725$. The red line is the 
asymptotic expression (\ref{sigmaasympt}).
{\it Right}: $\log \| \X u_n \|_{L^2([a_2,a_4])}$ versus $n$ for $\mu=5$ (red), $\mu=20$ (blue) and $\mu=100$ (green). 
The lines are the corresponding asymptotic expressions (\ref{asymptfn2}).
The plots show the last nine singular components.}
\label{fig1}
\end{center}
\end{figure}

\textbf{Acknowledgments.} RA was supported in part by a fellowship of the Research Foundation Flanders
(FWO) and AK was supported in part by NSF grant DMS-1211164. MD acknowledges support by the Strategic 
Research Programme SRP 10 of the Vrije Universiteit Brussel.

\section{Appendix}

This appendix gives the proof of equations (\ref{asymptfn2}) and (\ref{betamu}), which characterize  the asymptotic behavior as $n \rightarrow \infty$ of the integral
\begin{equation}
\|\X u_n\|_{L^2(\F)}= \left ( \int_{a_2}^{a_3-\mu} dx \, |u_n(x)|^2 \right )^{1/2}
\label{xfnmu}
\end{equation}
where $u_n$ is the $n$-th singular function of the truncated Hilbert transform with overlap $H_T$.
The proof relies heavily on the results from \cite{adk,a-k}. These findings rely on the fact that the operator $H_T$ shares an intertwining property with second-order differential operators $L_S$ and $\tilde{L}_S$ \cite{a-k}. It implies that the singular functions $u_n$ and $v_n$ are specially constructed solutions to
\begin{equation}\label{Lsfn}
P(x) f''(x) + P'(x) f'(x) =\left ( \lambda_n - 2 (x-\sigma)^2 \right ) f(x)
\end{equation}
where $\lambda_n \to \infty$ as $n \to \infty$ and $\lambda_n \to -\infty$ as $n \to -\infty$ (see \cite{adk,a-k} for details). Below we omit detailed references to these papers, to which the reader is referred  for all properties 
stated without explanation. 
\begin{lemma}\label{lemma1}
For small $\zeta > 0$ the asymptotic behavior as $n \to \infty$ of the singular function in $x \in (a_2+O(\varepsilon^{1+2 \zeta}), a_3-O(\varepsilon^{1+2 \zeta}))$ is
\begin{equation}
u_n(x) = \sqrt{\frac{2}{K_-}}  \frac{(-1)^{n+1}}{(P(x))^{1/4}} e^{-w_3(x)/\varepsilon} (1+ O(\varepsilon^{1/2- \zeta}))
 \label{gn6}
\end{equation}
\begin{equation}
\mbox{with  } w_3(x) = \int_{x}^{a_3}\frac{dt}{\sqrt{P(t)}}
\label{wdef}
\end{equation}
and where $\varepsilon$ is related to $n$ by 
\begin{equation}
\varepsilon = \frac{K_-}{n \pi} (1 + O \big( n^{-1/2+\zeta} \big)).
\label{discrete-eps}
\end{equation}
\end{lemma}

\begin{proof}
In the interval $(a_1+O(\varepsilon^{1+2 \zeta}), a_2-O(\varepsilon^{1+2 \zeta}))$,
where $\varepsilon$ is related to $n$ by (\ref{discrete-eps}), the WKB approximation of the singular function $v_n$ can be written as (see eqns (5.12) and (8.1) in \cite{adk})
\begin{align}
v_n(x) &= \frac{1}{ \sqrt{2 K_-}}\frac{1}{(-P(x))^{1/4}} {\Big \{ } [e^{i w_1(x)/\varepsilon - i \pi/4} + e^{-i w_1(x)/\varepsilon + i \pi/4}] (1 + O(\varepsilon^{1/2- \zeta})) \nonumber \\
&+ (1/i) [e^{i w_1(x)/\varepsilon -  i \pi/4} - e^{-i w_1(x)/\varepsilon + i \pi/4}] O(\varepsilon^{1/2- \zeta}) {\Big \} }
\label{gna1a2}
\end{align}
with
\begin{equation*}
w_1(x) = \int_{a_1}^x \frac{dt}{\sqrt{-P(t)}}.
\end{equation*}
Let $v$ be the analytic continuation of $v_n$ from the interval $(a_1,a_2)$ to $\bar{\mathbb{C}} \backslash [a_2,a_4]$. Employing the uniqueness of the solution to the Riemann-Hilbert problem and the Plemelj-Sokhotski formulas, one can show (cf. eqns (4.2)-(4.8) in \cite{adk}) that
\begin{align*}
u_n(x) &= \sigma_n \frac{v(x+i0)-v(x-i0)}{2i}, \quad x \in (a_2,a_4)\backslash \{a_3\},\\
v_n(x) &= \frac{v(x+i0)+v(x-i0)}{2},  \quad x \in (a_1,a_3)\backslash \{a_2\}.
\end{align*}
Furthermore, we have that (cf. eqns (4.9), (4,10) in \cite{adk})
\begin{align}
u_n(x) &= \sigma_n \text{Im } v(x+i0),  \quad x \in (a_2,a_4)\backslash \{a_3\}, \label{ung}\\
v_n(x) &= \text{Re } v(x+i0),  \quad x \in (a_1,a_3)\backslash \{a_2\}. 
\end{align}
The analytic continuation of $-P(x)$ from $(a_1,a_2)$ to $(a_2,a_3)$ via the upper half plane is given by $-P(x) = e^{-i \pi} P(x)$. With this, the analytic continuation of $v$ via the upper half plane to $(a_2+O(\varepsilon^{1+2 \zeta}),a_3-\mu) \subset (a_2+O(\varepsilon^{1+2 \zeta}), a_3-O(\varepsilon^{1+2 \zeta}))$ can be expressed as
(see section 5.1 and Figure 2 in \cite{adk}, and \cite{k-t} for the uniform accuracy of this continuation):
\begin{align}
v(x+i0) &= \frac{1}{ \sqrt{2 K_-}}\frac{1}{(P(x))^{1/4}} e^{i \pi/4} {\Big \{ } [e^{i K_-/\varepsilon - i \pi/4} e^{i e^{i\pi/2} w_2(x)/\varepsilon} + e^{-i K_-/\varepsilon + i \pi/4} e^{-i e^{i\pi/2} w_2(x)/\varepsilon}]  
\nonumber \\ 
&\cdot (1 + O(\varepsilon^{1/2- \zeta}))\nonumber \\ 
&+ (1/i) [e^{i K_-/\varepsilon - i \pi/4} e^{i e^{i\pi/2} w_2(x)/\varepsilon} - e^{-i K_-/\varepsilon + i \pi/4} e^{-i e^{i\pi/2} w_2(x)/\varepsilon}]  O(\varepsilon^{1/2- \zeta}) {\Big \} }
\label{gn0}
\end{align}
with
\begin{equation*}
w_2(x) = K_+ - w_3(x)=\int_{a_2}^x \frac{dt}{\sqrt{P(t)}}.
\end{equation*}
Equation (\ref{gn0}) can be rewritten as:
\begin{align}
v(x+i0) &= \frac{1}{ \sqrt{2 K_-}}\frac{1}{(P(x))^{1/4}}  {\Big \{ } [e^{i K_-/\varepsilon } e^{-w_2(x)/\varepsilon} + e^{-i K_-/\varepsilon+i\pi/2} e^{w_2(x)/\varepsilon}]  
\nonumber \\ &\cdot (1 + O(\varepsilon^{1/2- \zeta}))
+ (1/i) [e^{i K_-/\varepsilon } e^{-w_2(x)/\varepsilon} - e^{-i K_-/\varepsilon+i\pi/2} e^{w_2(x)/\varepsilon}]  O(\varepsilon^{1/2- \zeta}) {\Big \} }
\label{gn2}
\end{align}
Noting that the modulus of all complex exponentials is equal to $1$ and that $e^{w_2(x)/\varepsilon} > e^{-w_2(x)/\varepsilon}$, 
one can regroup all  terms in $O(\varepsilon^{1/2- \zeta})$ in (\ref{gn2}) to:
\begin{equation}
 \frac{1}{(P(x))^{1/4}} e^{w_2(x)/\varepsilon} O(\varepsilon^{1/2- \zeta})
 \label{gn4}
\end{equation}
Therefore, (\ref{gn2}) becomes 
\begin{align}
v(x+i0) =& \frac{1}{ \sqrt{2 K_-}} \frac{1}{(P(x))^{1/4}}  \big( e^{-i K_-/\varepsilon + i \pi/2} e^{w_2(x)/\varepsilon} + e^{i K_-/\varepsilon} e^{-w_2(x)/\varepsilon} + e^{w_2(x)/\varepsilon} O(\varepsilon^{1/2- \zeta})\big)  \nonumber \\
=& \frac{1}{ \sqrt{2 K_-}} \frac{1}{(P(x))^{1/4}}  e^{w_2(x)/\varepsilon} \big( e^{-i K_-/\varepsilon + i \pi/2} + O(\varepsilon^{1/2- \zeta}) \big).
\label{gn2a}
\end{align}

Combining (\ref{ung}) and (\ref{gn2a}) one obtains,
\begin{equation}
u_n(x) = \frac{\sigma_n}{\sqrt{2 K_-}}  \frac{(-1)^{n+1}}{(P(x))^{1/4}} e^{w_2(x)/\varepsilon} (1+ O(\varepsilon^{1/2- \zeta}))
\end{equation}
where we have used $\cos (K_-/\varepsilon) =(-1)^{n+1}(1+ O(\varepsilon^{1/2- \zeta}))$ (see eq. (5.36) in \cite{adk}).
Finally, using the explicit expression (see (5.45) in \cite{adk}) for the asymptotics of the singular values, yields equation (\ref{gn6}) with $w_3(x)$ as in (\ref{wdef}).
This result holds for $x \in (a_2+O(\varepsilon^{1+2 \zeta}), a_3-O(\varepsilon^{1+2 \zeta}))$. 
\end{proof}
\begin{lemma}\label{lemma2}
For sufficiently large positive index $n$, the singular function $u_n$ is strictly monotonic on $(a_2, a_3)$.
\end{lemma}

\begin{proof}
The singular function $u_n(x)$ is a solution of (\ref{Lsfn}) for $x \in (a_2,a_4)\backslash \{a_3\}$, where $\lambda_n \to \infty$ as $n \to \infty$. By definition, $P(x) > 0$ for $x \in (a_2, a_3)$ and $P'(a_2) > 0$. Since $u_n(x)$ is bounded at $x=a_2$ and $u_n(a_2) \neq 0$ (see e.g. \cite{a-k}, Sec. 2), we may assume without loss of generality that $u_n(a_2) > 0$. Then if 
\begin{equation}
\lambda_n > \max_{a_2 \leq x \leq a_3} 2 (x-\sigma)^2
\label{minlambda}
\end{equation}
the RHS of (\ref{Lsfn}) is positive at $a_2$ and since $P(a_2)=0$, $u''_n(a_2)$ is finite and $P'(a_2) > 0$, one concludes that $u'_n(a_2) > 0$. We will now show that
$u'_n(x) > 0$ for $x \in (a_2, a_3)$.  

Let $\tilde x$ be the first point after $a_2$ for which $u_n'(\tilde x)=0$. Then, by construction we have $u_n(\tilde x)>0$ and $u_n''(\tilde x) \leq 0$. With this, evaluating (\ref{Lsfn}) at $x=\tilde x$ results in a contradiction since the LHS is non-positive whereas the RHS is positive. Hence, $u_n'(x)>0$ for all $x \in (a_2,a_3)$, i.e., $u_n$ is strictly monotonic.
\end{proof}

\begin{proof}[Proof of (\ref{asymptfn2}),(\ref{betamu})] 
Split the integral (\ref{xfnmu}) into
\begin{equation}
\|\X u_n\|_{L^2(\F)}^2= I_a+I_b=\int_{a_2+\rho}^{a_3-\mu} dx \, |u_n(x)|^2 + \int_{a_2}^{a_2+\rho} dx \, |u_n(x)|^2
\label{chinIaIb}
\end{equation}
with $\rho = O(\varepsilon^{1+2 \zeta})$ and $\mu>0$ such that the integration interval for $I_a$ is contained in the domain of validity of the WKB approximation for
$n > N_\mu$.  Using Lemma \ref{lemma1}, the first term is 
\begin{align}
I_a &= \int_{a_2+\rho}^{a_3-\mu} dx \, |u_n(x)|^2 = \frac{\varepsilon}{K_-}  (e^{-2 w_3(a_3-\mu)/\varepsilon} -e^{-2 w_3(a_2+ \rho)/\varepsilon}  ) (1+ O(\varepsilon^{1/2- \zeta}))
\nonumber \\
&= \frac{\varepsilon}{K_-}  e^{-2 w_3(a_3-\mu)/\varepsilon} (1-e^{2/\varepsilon[w_3(a_3-\mu)-w_3(a_2+\rho)]})  (1+ O(\varepsilon^{1/2- \zeta})) \nonumber \\
&= \frac{\varepsilon}{K_-}  e^{-2 w_3(a_3-\mu)/\varepsilon}  (1+ O(\varepsilon^{1/2- \zeta})), \nonumber
\end{align}
where in the second line the first factor in parantheses gets absorbed in the error term since $w_3(a_3-\mu) - w_3(a_2+ \rho)$ is negative and strictly bounded away from $0$.
Consider now the second term $I_b$. For sufficiently small $\varepsilon=1/\sqrt{\lambda_n}$, $u_n$ is monotonic on $(a_2,a_3)$ by Lemma \ref{lemma2}. Thus, 
\begin{equation}
I_b=\int_{a_2}^{a_2+\rho} dx \, |u_n(x)|^2 \leq \frac{\rho}{a_3-\mu-a_2-\rho} \, \int_{a_2+ \rho}^{a_3-\mu} dx \, |u_n(x)|^2 \leq O(\varepsilon^{1+2 \zeta}) \, I_a
\label{Ibbound}
\end{equation}
so that the contribution of $I_b$ can be absorbed in the error term of equation (\ref{I1}). We finally obtain
\begin{align}
\|\X u_n\|_{L^2(\F)}= (I_a+I_b)^{1/2} &= \left ( \frac{\varepsilon}{K_-}  e^{-2 w_3(a_3-\mu)/\varepsilon} (1+ O(\varepsilon^{1/2- \zeta})) \right )^{1/2} \nonumber \\
&= \frac{1}{\sqrt{n \pi}} e^{-\bb n} \, (1+ O(n^{-1/2+ \zeta}))
\label{chinall2}
\end{align}
\end{proof}

\end{document}